\documentclass{amsart}
\usepackage{amsmath}
\usepackage{amssymb}

\numberwithin{equation}{section}

\newtheorem{theorem}{Theorem}[section]
\newtheorem{proposition}[theorem]{Proposition}
\newtheorem{corollary}[theorem]{Corollary}
\newtheorem{lemma}[theorem]{Lemma}

\theoremstyle{definition}

\newcommand{\bd}{{B_d}}

\newcommand{\neva}{N}

\newcommand{\aream}{A}
\newcommand{\meal}{\sigma_d}

\newcommand{\dbr}{\mathcal{H}}

\newcommand{\er}{\varepsilon}
\newcommand{\za}{\zeta}

\newcommand{\ph}{\varphi}
\newcommand{\cph}{C_\varphi}
\newcommand{\om}{\omega}
\newcommand{\hol}{\mathcal{H}ol}
\newcommand{\Dbb}{\mathbb D}
\newcommand{\Tbb}{\mathbb T}

\newcommand{\knl}{k}

\newcommand{\Cbb}{\mathbb C}
\newcommand{\Nbb}{\mathbb N}
\newcommand{\spd}{\partial B_d}

\begin{document}

\title[Composition operators]{Composition operators between de Branges-Rovnyak and Hardy spaces}
\author{Evgueni Doubtsov}
\address{St.~Petersburg Department
of Steklov Mathematical Institute, Fontanka 27, St.~Petersburg 191023, Russia}
\email{dubtsov@pdmi.ras.ru}

\author{Andrei V. Vasin}
\address{Admiral Makarov State University of Maritime and Inland Shipping,
Dvinskaya st.~5/7, St.~Petersburg 198035, Russia}
\email{andrejvasin@gmail.com}



\begin{abstract}
Let $d\ge 1$ and $\varphi: B_d\to\mathbb{D}$ be a holomorphic function,
where $B_d$ denotes the open unit ball of $\mathbb{C}^d$ and $\mathbb{D} = B_1$.
Let $b: \mathbb{D} \to \mathbb{D}$ be a holomorphic function and
$\mathcal {H}(b)$  denote the corresponding de Branges--Rovnyak space.
We show that compactness of the composition operator $C_\varphi$
from $\mathcal {H}(b)$ to the Hardy space $H^2(B_d)$
is related to natural restrictions on the Nevanlinna counting functions
of the slice-functions $\varphi_\zeta$, $\zeta\in \partial B_d$.
\end{abstract}

\maketitle

\section{Introduction}
For $d\ge 1$, let $\hol(\bd)$ denote the space of holomorphic functions in the unit ball
$\bd = \{z\in\Cbb^d: |z|<1\}$.
Let $\Dbb = B_1$ and $\Tbb = \partial \Dbb$.
Let $\meal$ denote the normalized Lebesgue measure on the unit sphere $\spd$.

For $0<p<\infty$, the classical Hardy space $H^p=H^p(\bd)$ consists of
$f\in \hol(\bd)$ such that
\[
\|f\|_{H^p}^p = \sup_{0<r<1} \int_{\spd} |f(r\za)|^p\, d\meal(\za) < \infty.
\]
As usual, the Hardy space $H^p(\bd)$, $p>0$, is identified with the space
$H^p(\spd)$ of the corresponding boundary values.

\subsection{Composition operators}
Consider a holomorphic mapping $\ph: \bd\to\Dbb$, $d\ge 1$.
It is well known that the composition operator $\cph: f\mapsto f\circ\ph$ maps $H^p(\Dbb)$
into $H^p(\bd)$, $p>0$. 
Therefore, it is natural to ask about characterizations
of those symbols $\ph$ for which
$\cph: H^p(\Dbb)\to H^p(\bd)$ is a compact operator.
For $d=1$, an explicit approach based on the Nevanlinna counting function
was developed in \cite{ShJoel87}.
A different answer, in terms of the Clark measures, was obtained in \cite{CM97}.

\subsection{The space $\dbr(b)$}
Each function $\psi \in L^\infty(\Tbb)$
generates the Toeplitz operator $T_\psi$ on $H^2(\Dbb)$
defined by
\[
T_\psi (f) = P_+ (\psi f),\quad f\in H^2(\Dbb),
\]
where $P_+$ is the Riesz orthogonal projector from $L^2(\Tbb)$ onto $H^2(\Tbb)$. 
To each holomorphic function $b: \Dbb \to \Dbb$, we associate the de Branges--Rovnyak space
\[
\dbr(b) = (I- T_b T_{\overline{b}})^{1/2} H^2,
\]
equipped with the inner product
\begin{equation}\label{e_ipH}
\langle(I- T_b T_{\overline{b}})^{1/2} f, (I- T_b T_{\overline{b}})^{1/2} g\rangle_b = \langle f, g \rangle_{H^2},
\end{equation}
making $\dbr(b)$ a Hilbert space. In formula \eqref{e_ipH}, $f$ and $g$ are taken in $H^2$ such
that $f, g \bot \mathrm{ker}(I- T_b T_{\overline{b}})^{1/2}$.

For further results about the spaces $\dbr(b)$,
see \cite{FM16} and \cite{Sa94}.

\subsection{Composition operators on de Branges--Rovnyak spaces $\dbr(b)$}
In the present work,
we use the Nevanlinna counting function 
to obtain necessary conditions and sufficient conditions
for the compactness of the operator $\cph: \dbr(b) \to H^2(\bd)$.

Let $\phi:\Dbb \to \Dbb$ be a holomorphic function.
Recall that the Nevanlinna counting function $\neva_\phi$ is defined by
\begin{equation}\label{e_neva}
\neva_{\phi}(w) = \sum_{z\in \Dbb:\, \phi(z)=w} \log\frac{1}{|z|}, \quad w\in\Dbb\setminus\{\phi(0)\},
\end{equation}
where each pre-image is counted according to its multiplicity.

For $\za\in\spd$ and a holomorphic mapping $\ph: \bd\to\Cbb$, $d\ge 1$, the slice-function $\ph_\za$ 
is defined as $\ph_\za(z) = \ph (z\za)$, $z\in\Dbb$.

The following theorem is the main result of the present paper.

\begin{theorem}\label{t_cph_tri}
Let $\ph: \bd\to\Dbb$, $d \ge 1$, be a holomorphic function, $\ph(0)=0$.
Let $b: \Dbb\to \Dbb$ be a holomorphic function such that $\dbr(b)$ is infinite dimensional.
Consider the following properties:
\begin{itemize}
  \item [(i)] for some $\gamma\in (0, 1/3)$,
  \begin{equation}\label{e_nevaGam0}
\int_{\spd}\neva_{\ph_\za}(z) \left( \frac{(1-|b(z)|)^\gamma}{1-|z|^2} \right)^2 \, d\meal(\za)\to 0 \quad \textrm{as}\ |z|\to 1-,
\end{equation}
  \item [(ii)] the operator $\cph: \dbr(b) \to H^2(\bd)$ is compact,
  \item [(iii)]
 \begin{equation}\label{e_neva0}
\int_{\spd}\neva_{\ph_\za}(z) \frac{1-|b(z)|^2}{1-|z|^2} \, d\meal(\za)\to 0 \quad \textrm{as}\ |z|\to 1-.
\end{equation} 
\end{itemize}
Then $\mathrm{(i)} \Rightarrow \mathrm{(ii)} \Rightarrow \mathrm{(iii)}$.
\end{theorem}

For $d=1$, Theorem~\ref{t_cph_tri} was obtained in \cite{FKM19}.
If $b$ is an inner function, then $\dbr(b)$ coincides with the model space $K_b = H^2 \ominus bH^2$.
Characterizations of the compact operators $\cph: K_b \to H^2(\bd)$ were earlier obtained in \cite{LM13} for $d=1$.
See also \cite{Dou25CMB} for the case $d\ge 1$.

To prove Theorem~\ref{t_cph_tri}, we adapt methods from \cite{FKM19}, \cite{LM13} and \cite{ShJoel87}
to the case of several variables.

\subsection*{Organization of the paper}
Auxiliary results are collected in Section~\ref{s_aux}. 
A necessary condition for the compactness of
$\cph: \dbr(b) \to H^2(\bd)$ is given in Section~\ref{s_nec}.
The corresponding sufficient conditions are obtained in Section~\ref{s_suff}.

\section{Auxiliary results}\label{s_aux}
\subsection{Littlewood--Paley identity and its generalizations}\label{ss_HL}

For $f\in H^2(\Dbb)$, the Little\-wood--Paley identity states that
\begin{equation}\label{e_LP}
\|f\|^2_{H^2(\Dbb)} = |f(0)|^2 + 2\int_{\Dbb} |f^\prime(w)|^2 \log\frac{1}{|w|}\, d\aream(w),
\end{equation}
where $\aream$ denotes the area measure on the disk $\Dbb$.

\subsubsection{Stanton's formula}
To study the composition operator generated by a holomorphic symbol $\phi:\Dbb \to \Dbb$,
J.~H.~Shapiro \cite{ShJoel87} used for $f\circ\phi$ an analog of the identity \eqref{e_LP}.
This analog is based on the Nevanlinna counting function $\neva_\phi$ defined by \eqref{e_neva}.
The following Stanton formula is the principal technical tool in Shapiro's argument. 

\begin{theorem}[\cite{ShJoel87}]\label{t_Stn_disk}
Let $\phi: \Dbb\to\Dbb$ be a holomorphic function.
Then
\begin{equation}\label{e_Stntn}
\|f\circ \phi\|^2_{H^2(\Dbb)} = |f(\phi(0))|^2 + 2\int_{\Dbb} |f^\prime(w)|^2
\neva_{\phi}(w) \, d\aream(w).
\end{equation}
\end{theorem}

\begin{corollary}\label{c_Stanton}
Let $\ph: \bd\to\Dbb$, $d\ge 1$, be  a holomorphic function.
Then
\begin{equation}\label{e_Stntn_d}
\|f\circ \ph\|^2_{H^2(\bd)} = |f(\ph(0))|^2 + 2\int_{\Dbb} |f^\prime(w)|^2
\left(\int_{\spd} \neva_{\ph_\za}(w)\,  d\meal(\za) \right)\, d\aream(w).
\end{equation}
\end{corollary}
\begin{proof}
Let $\za\in \spd$.
Applying Theorem~\ref{t_Stn_disk} with $\phi = \ph_\za$ and integrating with respect to  the normalized Lebesgue measure
$\meal$ on $\spd$, we obtain the required equality \eqref{e_Stntn_d}.
\end{proof}

\subsection{Subharmonic property for the Nevanlinna counting function}

\begin{proposition}[{\cite[Section 4.6]{ShJoel87}}]\label{p_subharm}
Let $w\in\Dbb$ and $\phi: \Dbb\to\Dbb$ be a holomorphic function.
Let $\Delta\subset \Dbb$ be a disk centered at $w$ and such that $\phi(0)\notin \Delta$.
Then
\begin{equation}\label{e_subharm_1}
\neva_{\phi}(w) \le \frac{1}{\aream(\Delta)}\int_{\Delta} \neva_\phi(z)\, d\aream(z).
\end{equation}
\end{proposition}

\begin{corollary}\label{c_subharm}
Let $w\in\Dbb$ and $\ph: \bd\to\Dbb$ be a holomorphic function.
Let $\Delta\subset\Dbb$ be a disk centered at $w$ and such that $\ph(0)\notin \Delta$.
Then
\begin{equation}\label{e_subharm}
\int_{\spd} \neva_{\ph_\za}(w)\,d\meal(\za) \le \frac{1}{\aream(\Delta)}\int_{\Delta} 
\left(\int_{\spd} \neva_{\ph_\za}(z)\,  d\meal(\za) \right)\, d\aream(z).
\end{equation}
\end{corollary}
\begin{proof}
Let $\za\in \spd$.
Using Proposition~\ref{p_subharm} with $\phi = \ph_\za$,
integrating with respect to $\meal$ and applying Fubini's theorem, we obtain the required
inequality \eqref{e_subharm}.
\end{proof}

\subsection{Spaces boundedly contained in $H^2$}
Recall
that a Hilbert space $E$ is said to be boundedly contained in $H^2(\Dbb)$ if 
$E \subset H^2(\Dbb)$ and $\|f\|_2 \lesssim \|f\|_E$
for all $f \in E$.

\begin{lemma}[see, for example, {\cite[Lemma 10]{FKM19}}]\label{l_10}
Let $E$ be boundedly contained in $H^2(\Dbb)$. Then the subspace
\[
E_n = \{f \in E: f \ \textrm{has zero of order $n$ at the origin} \}
\]
is closed in $E$. Moreover, the subspace $E_n^\bot$ is finite dimensional.
\end{lemma}

\subsection{Reproducing kernels for $\dbr(b)$}
It is well known that $\dbr(b)$ is contractively contained in $H^2$ and thus it is a reproducing kernel Hilbert space. The reproducing kernel of $\dbr(b)$ at $\lambda\in \Dbb$ is given by the following formula
\[
k^b_\lambda (z) =
\frac{1-b(z) \overline{b(\lambda)}}{1- z\overline{\lambda}}, \quad
z, \lambda \in \Dbb.
\]
It is known that
\[
\|k^b_\lambda\|_b = \left(\frac{1- |b(\lambda)|^2}{1-|\lambda|^2}\right)^{1/2}, \quad
\lambda \in \Dbb.
\] 
Hence, the normalized reproducing kernel is
\[
\widetilde{k}^b_\lambda (z) =
\frac{1-b(z) \overline{b(\lambda)}}{1- z\overline{\lambda}} \left(\frac{1-|\lambda|^2}{1- |b(\lambda)|^2}\right)^{1/2}, \quad
z, \lambda \in \Dbb.
\]

Let $D_\er(w) = \{z\in\Dbb: |z-w| < \er|1- z\overline{w}|\}$, that is,
let $D_\er(w)$ denote the pseudo-hyperbolic $\er$-disk centered at  $w\in \Dbb$.

\begin{lemma}[{\cite[Lemma~1]{LM13} and \cite[Lemma~8]{FKM19}}]\label{l_LM}
Let $b: \Dbb \to\Dbb$ be a holomorphic function.
Let $\{w_n\}_{n=1}^\infty$ be a sequence in $\Dbb$ such that $|w_n|\to 1$ as $n\to \infty$ and
\begin{equation}\label{e_LM}
|b (w_n)| < a
\end{equation}
for some parameter $a\in (0,1)$. Then
\begin{itemize}
\item[(i)] $\widetilde{\knl}_{w_n}^b \overset{w^\ast}\longrightarrow 0$ as $n\to \infty$;
\item[(ii)] there exist $\er>0$, $C>0$ and $n_0\in\Nbb$ such that
\[
|(\knl_{w_n}^b)^\prime(z)| \ge \frac{C}{(1-|w_n|^2)^2}, \quad z\in D_\er(w_n),
\]
\end{itemize}
for all $n\ge n_0$.
\end{lemma}

\subsection{Bernstein inequalities}
Let $1 < p < 2$ and let $q$ be the conjugate exponent of $p$. Also, let
$\rho(\za) = 1- |b(\za)|^2$, $\za\in\Tbb$, and
\[
\mathfrak{K}^\rho_\lambda(z) = \overline{b(\lambda)}
\frac{
2 b(z)\overline{b(\lambda)}- b^\prime(z) \overline{b^\prime(\lambda)}}{
1-z\overline{\lambda}}, \quad z, \lambda \in\Dbb. 
\]
Finally, put
\[
\om_p(z) = \min
\left\{
\|(k^b_z)^2\|^{-p/(p+1)}_q, 
\|\rho^{1/q} \mathfrak{K}^\rho_z\|^{-p/(p+1)}_q
\right\}, \quad z\in \Dbb.
\]

Recall that a finite positive Borel measure $\mu$ on $\Dbb$ is called a Carleson measure
if there exists a constant $C>0$ such that
\[
\mu(S(\za, h)) \le C h
\]
for all $h>0$ and $\za\in \Tbb$, where $S(\za, h) = \{z\in\Dbb: |z-\za|< h\}$.

\begin{theorem}[{\cite[Theorem 4.1]{BFM10}}]\label{t_2}
Let $\mu$ be a Carleson measure on $\Dbb$, let $1<p<2$ and let $b: \Dbb \to \Dbb$ be a holomorphic function.
Then there exists a constant $C= C(\mu, p)$  such that
\[
\|f^\prime \om_p\|_{L^2(\mu)} \le C\|f\|_b, \quad f\in\dbr(b).
\]
\end{theorem}

Also, by Lemma~3.5 from \cite{BFM10}, there exists a constant $A>0$ such that
\begin{equation}\label{e_om_p_A}
\om_p(z)\ge A \frac{1-|z|}{\left(1- |b(z)| \right)^{\frac{p}{q(p+1)}}}, \quad z \in \Dbb.
\end{equation}

\section{A necessary condition}\label{s_nec}

In this section, we prove that (ii) $\Rightarrow$ (iii) in Theorem~\ref{t_cph_tri}.

Assume that $\cph: \dbr(b)\to H^2(\bd)$ is compact but (iii) does not hold.
Then there exists a sequence $\{w_n\}_{n=1}^\infty \subset \Dbb$ such that $|w_n|\to 1$ and
\begin{equation}\label{e_neva_Not0}
\int_{\spd}\neva_{\ph_\za}(w_n) \frac{1-|b(w_n)|^2}{1-|w_n|^2}\, d\meal(\za)\ge c> 0.
\end{equation}
By the Littlewood subordination principle \cite{L25}, $\neva_{\ph_\za}(w_n) \le \log\frac{1}{|w_n|}$,
hence, \eqref{e_neva_Not0} implies \eqref{e_LM} for some $a\in (0,1)$.
Therefore, Lemma~\ref{l_LM} is applicable.

Sequentially applying \eqref{e_Stntn_d}, Lemma~\ref{l_LM}(ii) and Corollary~\ref{c_subharm}, we obtain
\begin{equation}\label{e_rker_estim}
\begin{aligned}
\|C_\ph \widetilde{\knl}_{w_n}^b \|^2_{H^2(\bd)}
    &\ge \frac{1}{\|\knl_{w_n}\|^2} \int_{\Dbb} |\knl^\prime_{w_n}(z)|^2
\left(\int_{\spd} \neva_{\ph_\za}(z)\,  d\meal(\za) \right)\, d\aream(z)\\
    &\ge \int_{D(w_n, \er)} \frac{C}{(1- |w_n|^2)^{3}}
\left(\int_{\spd} \neva_{\ph_\za}(z)\,  d\meal(\za) \right)\, d\aream(z)\\
    &\ge \frac{C_\er}{1- |w_n|^2} \int_{\spd} \neva_{\ph_\za}(w_n)\,d\meal(\za).
\end{aligned}
\end{equation}
Now, recall that $C_\ph$ is a compact operator, 
thus, $\|C_\ph \widetilde{\knl}_{w_n}^b(z)\| \to 0$ by Lemma~\ref{l_LM}(i).
Therefore, \eqref{e_rker_estim} contradicts \eqref{e_neva_Not0}.
The proof of the implication (ii) $\Rightarrow$ (iii) is finished.

\section{Sufficient conditions}\label{s_suff}

In this section, we prove that (i) $\Rightarrow$ (ii) in Theorem~\ref{t_cph_tri}.

\begin{proposition}\label{p_suff}
Let $\ph: \bd \to \Dbb$ be a holomorphic function, $\ph(0)=0$.
Let $b:\Dbb \to \Dbb$ be a holomorphic function such that $\dbr(b)$ is infinite dimensional.
Assume that there exists $p\in (1,2)$ such that
\begin{equation}\label{e_om2Neva0}
\om_p^{-2}(z) \int_{\spd} N_{\ph_\za} (z)\, d\meal(\za) \to 0\quad \textrm{as}\ |z|\to 1-.
\end{equation}
Then the operator $\cph: \dbr(b) \to H^2(\bd)$ is compact.
\end{proposition}
\begin{proof}
Let $\dbr_n (b) = \{f\in\dbr(b): f
\textrm{\ has zero of order $n$ at the origin}  \}$.
Recall that $\dbr(b)$ is contractively contained in $H^2(\Dbb)$, thus,
by Lemma~\ref{l_10}, we can define the corresponding orthogonal projector
\[
P_n: \dbr(b) \to \dbr_n(b).
\]
By Lemma~\ref{l_10}, the orthogonal complement of $\dbr_n(b)$ is finite dimensional.
We claim that
\begin{equation}\label{e_Pn0}
\|\cph P_n\|_{\dbr(b) \to H^2(\bd)} \to 0\quad\textrm{as\ } n\to \infty.
\end{equation}
Then the operator $\cph$ is compact, since it is approximable
by finite rank operators $\cph (I-P_n)$.

To verify property \eqref{e_Pn0}, fix $\er>0$ and $f\in \dbr(b)$, $\|f\|_{b} = 1$.
Put $g_n = P_n f$, then we have $\|g_n\|_{b} \le 1$.
Also, given $R<1$, there exists $n(R, \er)\in\Nbb$ independent of $f$ and such that
\begin{equation}\label{e_g_er}
|g_n(w)| < \er,\ |g^\prime(w)|< \er\quad \textrm{for}\ |w|<R\
\textrm{and for all}\ n>n(R, \er).
\end{equation}
Indeed, we have $g_n \in \dbr_n(b)$.
Thus, the Taylor series of $g_n$ at the origin has the following form:
\[
g_n(w) = \sum_{k=n}^\infty \frac{g_n(0)}{k!} w^k.
\]
Observe that
\[
\left| \frac{g_n(0)}{k!}\right| = |\langle g_n, w^k \rangle_2|
\le \|g_n\|_2 \|w^k\|_2 = \|g_n\|_2 \le \|g_n\|_b \le \|f\|_b =1.
\]
Thus, for $|w|\le R$, we obtain
\[
|g_n(w)| \le \sum_{k=n}^\infty |w|^k = \frac{|w|^n}{1-|w|} \le \frac{R^n}{1-R}.
\]
Therefore, there exists $n_1(R, \er)$ such that $|g_n| < \er$
for all $n > n_1(R, \er)$.

Also, we have
\[
g_n^\prime (w) = \sum_{k=n}^\infty \frac{g_n(0)}{k!} k w^{k-1}.
\]
Thus,
\[
|g_n^\prime (w)| \le \sum_{k=n}^\infty k |w|^{k-1} \le \sum_{k=n}^\infty k R^{k-1}.
\]
Therefore, there exists $n_2(R, \er)$ such that $|g_n^\prime| < \er$
for all $n > n_2(R, \er)$.
Hence, \eqref{e_g_er} holds.

Now, using Stanton's formula \eqref{e_Stntn_d}, we have
\[
\begin{split}
   \frac{1}{2} \|\cph P_n f \|_2^2 &= \int_{|z|<R} |g_n^\prime|^2 \int_{\spd} N_{\ph_\za}(z)\, d\meal(\za) dA(z)\\
     &\quad + \int_{R\le |z|< 1} |g_n^\prime|^2 \int_{\spd} N_{\ph_\za}(z)\, d\meal(\za) dA(z)\\
     &\le \sup_{|z|<R} |g_n^\prime(z)|^2  \cdot \int_{|z|<R}  \int_{\spd} N_{\ph_\za}(z)\, d\meal(\za) dA(z)\\
     &\quad + \sup_{R\le |z|< 1} \left( \om_p^{-2}(z) \int_{\spd} N_{\ph_\za}(z)\, d\meal(\za) \right) 
      \cdot \int_{R\le |z|< 1} |g_n^\prime|^2 \om_p^2(z)\\
     &= I_1 + I_2.
\end{split}
\]
Below we estimate $I_2$ and $I_1$.

Applying \eqref{e_om2Neva0}, select $R<1$ such that
\[
\sup_{R\le |z|< 1} \om_p^{-2}(z) \int_{\spd} N_{\ph_\za}(z)\, d\meal(\za) < \er.
\]
Next, the area measure on $\Dbb$ is a Carleson measure and $\|g_n\|_b \le 1$,
hence, by Theorem~\ref{t_2}, there exists a constant $C>0$ such that
\begin{equation}\label{e_Carle}
\int_{\Dbb} |g_n^\prime|^2 \om_p^2(z)\, dA(z) \le C
\end{equation}
for all $n\in \Nbb$. Thus,
\[
I_2 \le C\er
\]
for all $n\in \Nbb$.

Now, we consider $I_1$. We have
\[
\sup_{|z|<R} |g_n^\prime(z)|^2  < \er
\]
for all $n> n(R, \er)$. 
Also, by Littlewood's inequality,
\[
|N_{\ph_\za}(z)| \le \log\frac{1}{|z|}, \quad z\in\Dbb\setminus \{0\},
\]
since $\ph_\za(0)=0$ for all $\za\in\spd$.
Hence,
\[
\int_{\spd} \int_{|z|<R} N_{\ph_\za}(z)\, dA(z)\, d\meal(\za) \le C.
\]
Therefore,
\[
I_1 \le C\er
\]
for all sufficiently large $n$.

In sum, we obtain $I_1 + I_2 \le C \er$ for all sufficiently large $n$.
Thus, for all $f\in\dbr(b)$ with $\|f\|_b =1$, we have $\|\cph P_n f\|_2^2 \le C\er$ 
with an absolute constant $C>0$ and for all sufficiently large $n$.
Hence, \eqref{e_Pn0} holds. The proof of the proposition is finished.
\end{proof}

\begin{corollary}\label{c_suff}
Let $\ph: \bd \to \Dbb$ be a holomorphic function, $\ph(0)=0$.
Let $b:\Dbb \to \Dbb$ be a holomorphic function such that $\dbr(b)$ is infinite dimensional.
Assume that there exists $\gamma\in (0, \frac{1}{3})$ such that
\begin{equation}\label{e_gamNeva0}
\left(\frac{(1-|b(z)|)^\gamma}{1-|z|^2} \right)^2 \int_{\spd} N_{\ph_\za} (z)\, d\meal(\za) 
\to 0\quad \textrm{as}\ |z|\to 1-.
\end{equation}
Then the operator $\cph: \dbr(b) \to H^2(\bd)$ is compact.
\end{corollary}
\begin{proof}
Put $p = \frac{1+\gamma}{1-\gamma}$. Observe that $p\in (1,2)$.
Let $q$ denote the conjugate exponent of $p$.
Then
\[
\gamma = \frac{p-1}{p+1} = \frac{p}{q(p+1)}.
\]
By \eqref{e_om_p_A}, we have
\[
\om_p (z) \ge A\frac{1-|z|}{(1-|b(z)|)^\gamma}.
\]
Therefore, \eqref{e_gamNeva0} implies that
\[
\om_p^{-2}(z) \int_{\spd} N_{\ph_\za}(z)\, d\meal(\za) \to 0
\]
as $|z|\to 1$. Applying Proposition~\ref{p_suff}, we conclude that the operator
$\cph: \dbr(b) \to H^2(\bd)$ is compact.
\end{proof}

Finally, observe that Corollary~\ref{c_suff} coincides with the implication (i)$\Rightarrow$(ii).
Thus, the proof of Theorem~\ref{t_cph_tri} is finished.

\providecommand{\bysame}{\leavevmode\hbox to3em{\hrulefill}\thinspace}
\providecommand{\MR}{\relax\ifhmode\unskip\space\fi MR }
\providecommand{\MRhref}[2]{%
  \href{http://www.ams.org/mathscinet-getitem?mr=#1}{#2}
}
\providecommand{\href}[2]{#2}

\end{document}